\documentclass[reqno]{amsart}

\usepackage{amsfonts,mathrsfs} 
\usepackage{amsmath}
\usepackage{amssymb}
\usepackage{amsthm}
\usepackage{graphics}
\usepackage{tikz}
\usetikzlibrary{arrows}
\usepackage{tikz-cd}
\usepackage{subcaption}

\newcommand{ \N }{\mathbb{N}}
\newcommand{ \Z }{\mathbb{Z}}
\newcommand{ \Q }{\mathbb{Q}}
\newcommand{ \R }{\mathbb{R}}
\newcommand{ \C }{\mathbb{C}}

\newcommand{ \Chat }{\widehat{\mathbb{C}}}
\newcommand{ \D }{\mathbb{D}}

\newcommand{ \A }{\mathcal{A}}
\newcommand{ \E }{\mathcal{E}}
\newcommand{ \Esc }{\mathcal{E}}

\newcommand{ \Sone }{\mathcal{S}_{1}}
\newcommand{ \Stwo }{\mathcal{S}_{2}}
\newcommand{ \Sn }{\mathcal{S}_{p}}
\newcommand{ \phyp }{\widehat{\mathcal{H}}_{0}}

\newcommand{ \bott }{\mathfrak{B}}
\newcommand{ \Tess }{\mathbf{Tes}}
\newtheorem*{mthm}{Main Theorem}
\newtheorem{theorem}{Theorem}[section]
\newtheorem{lemma}[theorem]{Lemma}
\newtheorem{cor}[theorem]{Corollary}
\theoremstyle{definition}
\newtheorem{defn}[theorem]{Definition}
\theoremstyle{remark}
\newtheorem*{Rem}{Remark}
\newtheorem*{Conj}{Conjecture}

\title[Relations between escape regions]{Relations between Escape Regions in the Parameter Space of Cubic Polynomials}

\author{Araceli Bonifant}
\address{Department of Mathematics and Applied Mathematical Sciences, University of Rhode Island, Kingston, Rhode Island}
\email{bonifant@uri.edu}
\author{Chad Estabrooks}
\address{Department of Mathematics and Physics, Belmont Abbey College, Belmont, North Carolina}
\email{chadestabrooks@bac.edu}
\author{Thomas Sharland}
\address{Department of Mathematics and Applied Mathematical Sciences, University of Rhode Island, Kingston, Rhode Island}
\email{tsharland@uri.edu}
\date{\today}

\begin{document}

\begin{abstract}
 We describe a topological relationship between slices of the parameter space of cubic maps. In the paper \cite{CP1}, Milnor defined the curves $\Sn$ as the set of all cubic polynomials with a marked critical point of period~$p$. In this paper, we will describe a relationship between the boundaries of the connectedness loci in the curves $\Sone$ and $\Stwo$. 
 
\end{abstract}

\maketitle

\section{Introduction}

The study of parameter spaces of polynomials is of fundamental importance in complex dynamics. The Orsay notes \cite{Orsay} contained, amongst other things, a detailed study of the space of quadratic polynomials. Since then, mathematicians have also turned their attention to other spaces of polynomials. In particular, researchers including Branner, Hubbard, Milnor, Roesch and Tan (among many others) have endeavored to understand the space of cubic polynomials. 

This present paper is a modest step in that direction. It reports the main results of the thesis of the second author \cite{Thesis}, and looks at the boundary of the connectedness locus $\mathcal{C}(\mathcal{P}(3))$ of cubic polynomials. A main tool for this analysis are the tessellations considered by Milnor and the first author in \cite{CP3}, which themselves are partially inspired by the puzzles (originally called patterns) discussed in \cite{ICP2} by Branner and Hubbard. In this paper, they remarked that knowledge of the escape regions may ``provide tools for creeping up to the cubic connectedness locus''. The main result of this work is to prove that the conformal isomorphism between escape regions in the parameter space of cubic polynomials can be extended (in a special case) to their boundaries.


\section{Preliminaries}

We quickly state some standard definitions from Complex Dynamics. For a more detailed treatment, the reader is encouraged to refer to \cite{Milnor}. Recall that a rational map $F$ on the Riemann sphere divides the sphere into two disjoint sets. The \emph{Julia set} $J(F)$, the closure of the set of periodic repelling points, is where the dynamics is ``chaotic'' and its complement, the \emph{Fatou set} is where the dynamics is tame. A connected component of the Fatou set is called a Fatou component.

\subsection{Notation} We quickly present the notation we will use in the paper. The external (respectively, internal) ray of angle $t$ for a map $F$ will be denoted by $r_e^F(t)$ (respectively, $r_i^F(t)$); its landing point is denoted by $\lambda_e^F(t)$ (respectively $\lambda_i^F(t)$). We will sometimes drop the superscript if it is clear which map is being referred to. In the parameter space, the external (respectively, internal) ray of angle $t$ is denoted by $R_e(t)$ (respectively, $R_i(t)$) and its landing point is $\Lambda_e(t)$ (respectively, $\Lambda_i(t)$). Thus we use lower case for the dynamical objects and upper case for the parameter objects. If $f_a \in \Sone$, the immediate basin of the fixed critical point is denoted by $\widehat{\A}_a$.

\subsection{Cubic Polynomials}

Following Milnor \cite{CP1}, we may parameterize the space of cubic polynomials by the pairs $(a,v) \in \C^2$ such that
\[
 f = f_{a,v}(z) = z^3 - 3a^2z + 2a^3 + v.
\]
Here the polynomial $f_{a,v}$ has critical points at $a$ and $-a$. Milnor observed that it is helpful to divide this parameter space up into ``slices''. Of particular interest are the slices $\mathcal{S}_p$, polynomials for which the (marked) critical point $a$ has exact period $p$ under iteration, while $-a$ is the ``free'' critical point. The point $2a$ has the same image under $f_{a,v}$ as the critical point $-a$; we call $2a$ the \emph{cocritical} point.

There are four types of hyperbolic components in the connectedness locus. A hyperbolic component is called
\begin{itemize}
 \item \emph{Type A} (Adjacent) if both critical points belong to the same periodic Fatou component.
 \item \emph{Type B} (Bitransitive) if the two critical points belong to different Fatou components in the same periodic cycle.
 \item \emph{Type C} (Capture) if one critical point belongs to a periodic cycle of Fatou components, and the second critical point eventually lands in this cycle.
 \item \emph{Type D} (Disjoint) if the polynomial has two distinct attracting periodic orbits, both of which attracts one critical point.
\end{itemize}

It is easy to see that one may parameterize the space $\Sone$ by a single complex parameter $a$, since then $f_a = f_{a,a}$ has a fixed critical point at $a$. There is a unique type $A$ component in $\Sone$ (see Figure~\ref{fig: S1S2}), which is called the \emph{principal hyperbolic component} and denoted by $\phyp$.
\subsection{Escape Regions}

Following \cite{CP2}, we define an escape region to be a connected component of $\Sn \setminus \mathcal{C}(\Sn)$, where $\mathcal{C}(\Sn) = \Sn \cap \mathcal{C}(\mathcal{P}(3))$. We will be interested in two particular escape regions: the unique escape region $\E_1$ in $\Sone$ and the basilica escape region $\E_2^B$ in $\Stwo$ (Figure~\ref{fig: S1S2}). For a polynomial $f \in \E_1$, all components of the filled Julia set which are not points are quasidisks. Similarly, in $\E_2^B$, all non-point components of the filled Julia set of a polynomial are quasiconformal copies of the Julia set for the ``basilica'' map $z \mapsto z^2-1$ (see Section~\ref{sec:bas}).

\begin{figure}[!ht]
    \centering
    \includegraphics[width=0.9\textwidth]{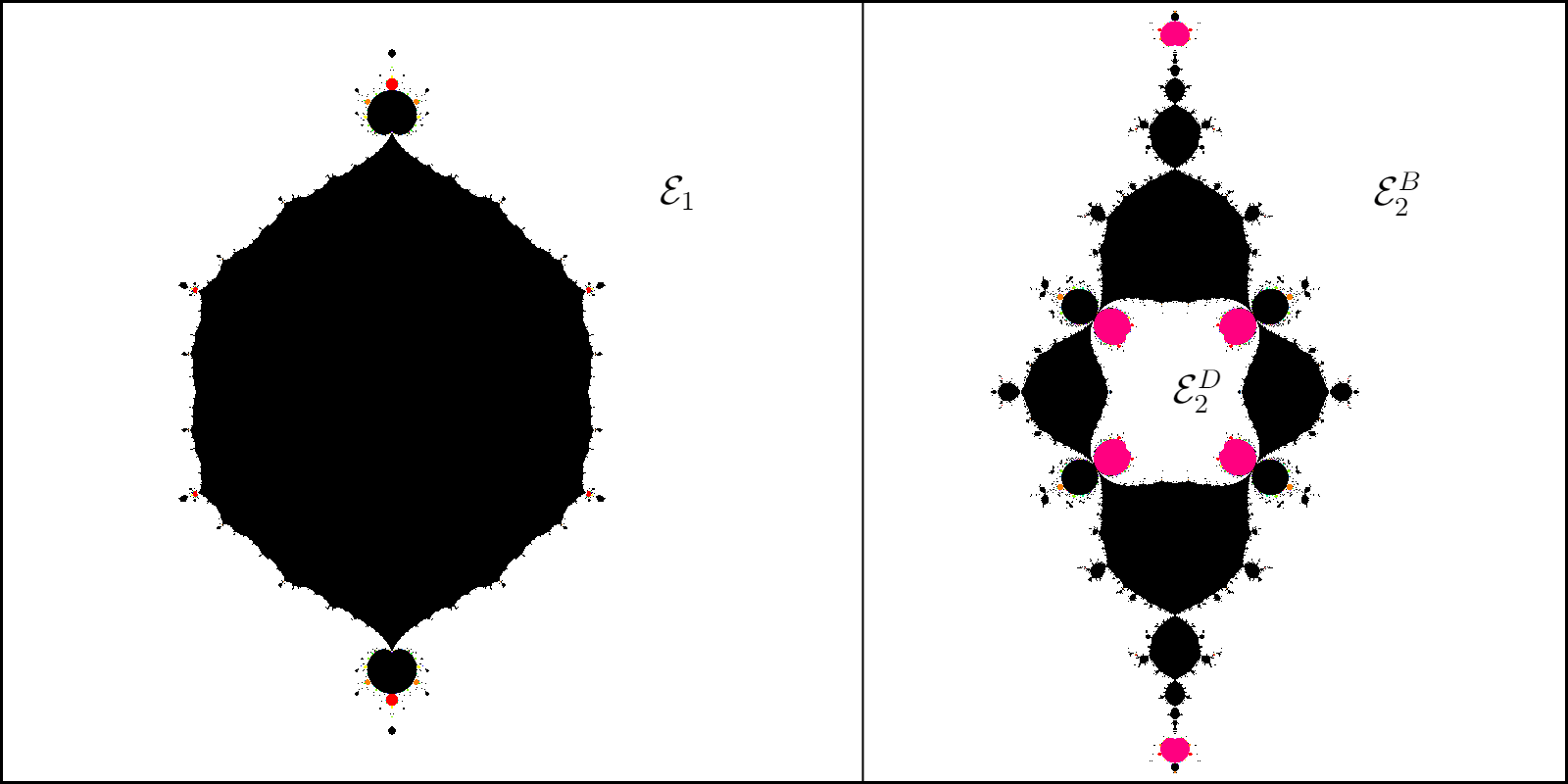}
    \caption{Connectedness loci (shaded) and escape regions (white) in the curves $\mathcal{S}_{1}$ and $\mathcal{S}_{2}$ (left and right, respectively) shown side by side. The principal hyperbolic component $\phyp$ in $\Sone$ is the black lemon-shaped region in the middle of the image.}
    \label{fig: S1S2}
\end{figure}

Let $f = f_{a,v} \in \Sn$. From B{\"o}ttcher's theorem (see e.g \cite[Theorem 9.1]{Milnor}), there is a maximal neighborhood $V$ of $\infty$ and a corresponding minimal $r \geq 1$ such that the map 
\[\bott_f: \Chat \setminus \overline{V} \to \Chat \setminus \overline{\D_{r}},\]
 conjugates $f$ to the map $z \mapsto z^3$. Adding the requirement that $\bott_f$ is tangent to the identity near $\infty$, there is a unique such map. Given an escape region $\E$ of multiplicity one\footnote{All escape regions considered in this paper are of multiplicity $1$. For more details, refer to \cite{CP2}.}, we may then define the conformal isomorphism
 \[
  \Phi_\E \colon \E \to \C \setminus \overline{\D}
 \]
by $\Phi_\E (f) = \bott_f(2a_f)$. Observe that these maps $\Phi_\E$ provide a ``coordinate system'' on $\E$: we denote by $(\rho,t)_\E$ the polynomial $f$ in $\E$ for which $\Phi_\E(f) = \rho e^{2 \pi i t}$. 

\subsection{Tessellations of $\mathcal{S}_p$}

For some $q \in \N$, let $A_q \subset \Q/\Z$ be the set of all arguments periodic of period $q$ under tripling. Such arguments have the form 
$$\frac{m}{3^{q}-1}$$
with $q$ minimal. We will say an argument $t$ is coperiodic (of period $q$) if one of $t + \frac13$ or $t - \frac13$ belongs to $A_q$.

The following definition comes from \cite{CP3}.

\begin{defn}
  Let $q \in \N$. We construct the period $q$ tessellation $\Tess_q(\overline{\mathcal{S}}_p)$ of $\overline{\mathcal{S}}_p$ as follows. The collection of all coperiod $q$ parameter rays in $\mathcal{S}_p$ decomposes $\mathcal{S}_p$ into a finite number of open sets $\mathcal{F}_k$, which we call the faces of the tessellation. The edges of the tessellation are the parameter rays of coperiod $q$, while the vertices split into two types: parabolic vertices are the (parabolic) landing points of the coperiodic rays and the ideal points are the elements of $\overline{\mathcal{S}}_p \setminus \mathcal{S}_p$.
\end{defn}

Of particular importance will be the tessellation $\Tess_2(\overline{\Stwo})$. Given a tessellation $\Tess_q(\overline{\mathcal{S}}_p)$, we may write
\[
  \Tess_q(\overline{\mathcal{S}}_p) = \left( \Tess^{(0)}_q(\overline{\mathcal{S}}_p), \Tess^{(1)}_q(\overline{\mathcal{S}}_p), \Tess^{(2)}_q(\overline{\mathcal{S}}_p) \right)
\]
 where $\Tess^{(0)}_q(\overline{\mathcal{S}}_p)$ is the set of vertices, $\Tess^{(1)}_q(\overline{\mathcal{S}}_p)$ is the set of edges and $\Tess^{(2)}_q(\overline{\mathcal{S}}_p)$ is the set of faces of $\Tess_q(\overline{\mathcal{S}}_p)$. 

We will also need the notion of an orbit portrait from \cite{CP3} (Compare \cite{POER}).

\begin{defn}
 Let $F \in \mathcal{S}_p$ be such that every ray of period $q$ lands at a (necessarily periodic) point of $J(F)$. The orbit portrait is the equivalence relation on the set of angles which have period $q$ under tripling defined by $\phi \sim \psi$ if and only if the dynamical rays $r_e(\phi)$ and $r_e(\psi)$ have a common landing point in $J(F)$.
 \label{def: orb port}
\end{defn}

Note that the orbit portrait is well-defined unless the map $F$ belongs to a parameter ray of coperiod $q$. In particular, it is well-defined for all faces of $\Tess_q(\overline{\mathcal{S}}_p)$ and so is well-defined on the connectedness locus $\mathcal{C}(\mathcal{S}_p)$. In fact we have the following.

%


\begin{theorem}{\cite[Theorem 3.12]{CP3}}
Let $\mathcal{F}_k$ be a face of $\Tess_q(\overline{\mathcal{S}}_p)$. Then for each $F \in \mathcal{F}_k$ and each angle $\theta_0$ which is periodic of period $q$ under tripling, the dynamical ray $r_e^F(\theta_0)$ lands at a repelling periodic point $z(F) \in J(F)$. Furthermore the correspondence $F \mapsto z(F)$ defines a holomorphic function $z \colon \mathcal{F}_k \to \C$. The orbit portrait for dynamic rays of period $q$ is the same for all maps $F \in \mathcal{F}_k$.   
\label{thm: BM}
\end{theorem}

\subsection{Conformal Isomorphisms between escape regions}
We will make use of the conformal isomorphisms
\[\Phi_1 = \Phi_{\mathcal{E}_1}:\E_{1} \to \C \setminus \overline{\D} \qquad\text{and}\qquad
\Phi_{2} = \Phi_{\mathcal{E}_2^B}:\E_{2}^{B} \to \C \setminus \overline{\D},\]
which are defined by mapping parameter rays to radial lines of corresponding argument. The composition 
\[\Psi = \Phi_{2}^{-1} \circ \Phi_{1}: \E_{1} \to \E_{2}^{B}\]
is then a conformal isomorphism between the two escape regions. Writing $(\rho,t)_1 = (\rho,t)_{\E_1}$ and $(\rho,t)_2 = (\rho,t)_{\E_2^B}$, we have in particular that $\Psi((\rho,t)_1) = (\rho,t)_2$. We will show that there exists a continuous extension $\widehat{\Psi}$ of $\Psi$, defined on the boundary points of the type $A$ and $C$ components of $\Sone$ which are the landing points of parameter rays of rational angle. We are interested in the points at which $\widehat{\Psi}$ is not injective.

Our main result is the following.

\begin{mthm}
 The map $\widehat{\Psi}$ is injective at all points $x$ except for the following.
 \begin{itemize}
  \item The point $x$ lies on the boundary of a type $C$ component and has internal argument which is a basilica angle.
  \item The point $x$ lies on the boundary of $\phyp$ and has internal argument which is a basilica angle distinct from $\frac{1}{3}$ or $\frac{2}{3}$.
 \end{itemize}

\end{mthm}

See Section~\ref{sec:bas} for the definition of a basilica angle. We remark that it is not immediate that the map $\widehat{\Psi}$ is well-defined. Indeed, the analogous map from $\E_1$ to $\E_2^D$ (the escape region of $\Stwo$ associated to the disk, see \cite{CP2}) has no continuous extension to the boundary of the type $A$ and $C$ components of $\Sone$. We prove that $\widehat{\Psi}$ is well-defined in Section~\ref{sec:preres} and prove the Main Theorem in Section~\ref{sec:main}. See Figure~\ref{fig:example} for an illustration of the main theorem.

\begin{figure}[h]
    \centering
    \begin{subfigure}[b]{0.465\textwidth}
        \includegraphics[width=\textwidth]{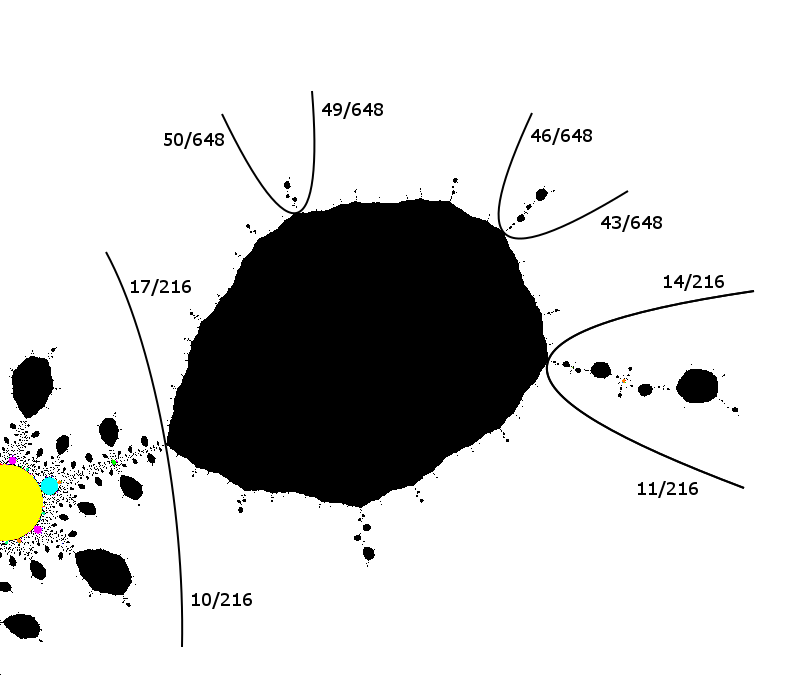}
        \caption{A close-up of part of $\Sone$.}
        \label{fig:magS1}
    \end{subfigure}
    \qquad
    \begin{subfigure}[b]{0.465\textwidth}
        \includegraphics[width=\textwidth]{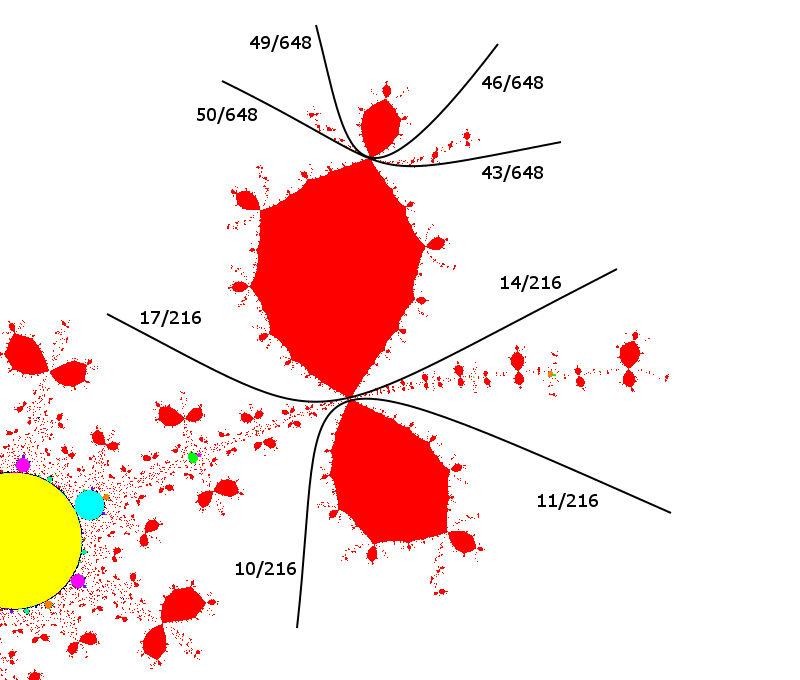}
        \caption{A close-up of part of $\Stwo$.}
        \label{fig:magS2}
    \end{subfigure}
    \caption{An example of the result of the Main Theorem, with some important parameter rays marked. Observe that the type $C$ components in $\Sone$ are collapsed down to (homeomorphic) copies of the basilica in $\Stwo$. }\label{fig:example}
\end{figure}

\subsubsection{Laminations}
We will also make use of the notion of a \emph{lamination}, originally constructed by Thurston \cite{GDR} to study the dynamics of rational maps on the sphere. A lamination is a set of chords (called \emph{leaves} of the lamination) in the closed unit disk satisfying the following. Let $L$ be the collection of leaves in a lamination.
\begin{itemize}
 \item The leaves are pairwise unlinked: that is, two different elements of $L$ are disjoint, except possibly at their endpoints;
 \item the union of $L$ is closed.
\end{itemize}

A \emph{gap} is the closure of the component of the complement of $\bigcup L$ in the closed unit disk. One can use laminations to model the structure of Julia sets of a polynomial $p$: two points $s$ and $t$ are identified in the lamination if and only if $\lambda_e^p(s) = \lambda_e^p(t)$. One may draw a lamination by drawing chords in the disk: the chord $\overline{st}$ is drawn in if and only if $s$ is identified with $t$ in the lamination, see Figure~\ref{fig:basandlam}.

\begin{figure}[h]
    \centering
    \begin{subfigure}[b]{0.4\textwidth}
        \includegraphics[width=\textwidth]{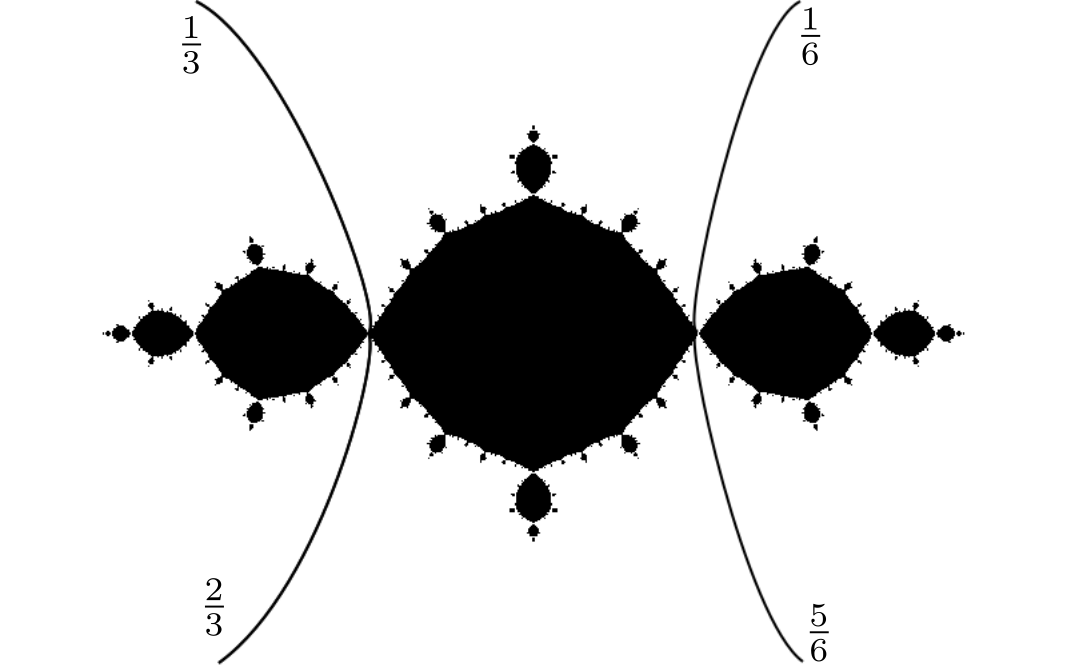}
        \caption{The Julia set of the basilica polynomial $z \mapsto z^2 - 1$.}
        \label{fig:bas}
    \end{subfigure}
    \qquad
    \begin{subfigure}[b]{0.4\textwidth}
        \includegraphics[width=\textwidth]{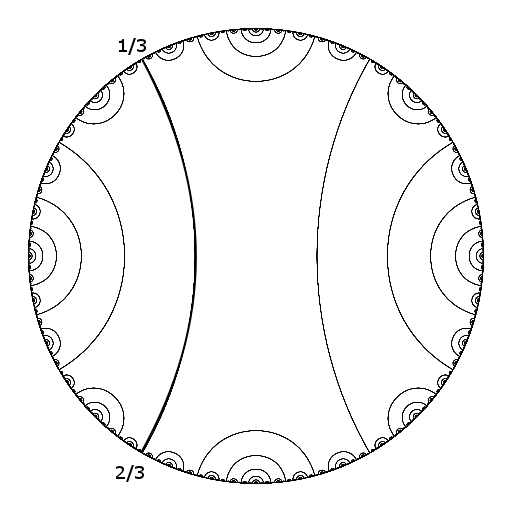}
        \caption{The lamination of the basilica}
        \label{fig:lam}
    \end{subfigure}
    \caption{The Julia set of the basilica and its associated lamination.}\label{fig:basandlam}
\end{figure}

\section{The Basilica polynomial}\label{sec:bas}

In this section we collect together results on the basilica polynomial $z \mapsto z^2 - 1$. These will be important when we start discussing the structure of the Julia sets in $\E_2^B$ in Section~\ref{sec:preres}. We begin with a simple fact about how the rays land on the Julia set of the basilica.

\begin{lemma}
 Let $t \in \R/\Z$ and suppose there exists $n \in \N$ such that $2^{n}t \equiv \frac{1}{3}$. Then there is a unique $\widetilde{t} \in \R/\Z$, $\widetilde{t} \neq t$, such that the external rays of angles $t$ and $\widetilde{t}$ land at the same point of the Julia set of the basilica.
\label{lem: ! bas pairs}
\end{lemma}

The proof follows from Proposition 4.3 of \cite{CNM}. This allows us to introduce the notion of a basilica angle.

\begin{defn}
An argument $t \in \R/\Z$ will be called a basilica angle if there is $n \in \N$ such that $2^{n}t \equiv \frac{1}{3} \pmod{1}$. Equivalently, $t$ is a basilica angle if the external  ray of angle $t$ for the basilica lands at a biaccessible point (that is, a point at which more than one external ray lands). If $s \neq t$ and the rays of angles $s$ and $t$ have the same landing point on the basilica, we say $\{s,t\}$ is a basilica pair.
\label{def: bas pair}
\end{defn}

We now want to transport the well-known facts about the basilica to polynomials belonging to the basilica escape region.  

\begin{lemma}
For a map $F = (r,t)_{2}^{B} \in \E_{2}^{B}$, there is a unique fixed point $\alpha$ in $J$ on the boundary of the Fatou components containing the marked critical point $a$, and its associated critical value $v$. 
\label{lem: ! fp alpha}
\end{lemma}

\begin{proof}
There is a neighborhood $U$ of the main basilica component of $K(F)$ such that the restriction of $F$ to $U$ is hybrid equivalent to the basilica map $z \mapsto z^{2} - 1$. 
\end{proof}

\begin{figure}[htb!]
    \centering
    \includegraphics[width=2.75 in]{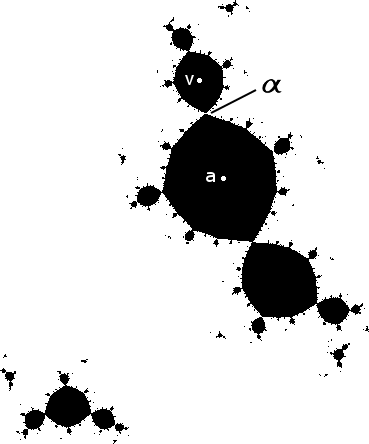}
    \caption{The filled Julia set of a map in $\E_{2}^{B}$ showing the placement of the fixed point $\alpha$.}
    \label{fig: avalpha}
\end{figure}

In particular, every polynomial in $\E_2^B$ has a distinguished fixed point. We will refer to this distinguished fixed point as $\alpha$.


\section{Preliminary results}\label{sec:preres}

In this section, we build a portfolio of results that will allow us to tackle the proof of the main theorem in Section~\ref{sec:main}.


\subsection{Behavior of the Conformal Isomorphism $\Psi$}

We would like to now describe the relationship between $J(f)$ and $J(\Psi(f))$ for $f \in \Sone$. Let $f \in \E_1 \cup \E_2^B$. Since $f \in R_e(t)$ if and only if $r_e(t)$ lands at the cocritical point of $f$, we see that the rays of angle $t - \frac13$ and $t + \frac13$, along with their landing point $-a$, split the plane into two regions. We denote these regions by $U_0$, which contains the critical point $a$, and $U_1$. Now observe that any ray landing on the main component of the filled Julia set $K(f)$ must have argument in the arc $\left( t - \frac13,t + \frac13 \right)$ (see Figure \ref{fig: creq updated}), and so must be contained in $U_0$. We state this as a lemma.

\begin{lemma}
Any pair of rays landing at $\alpha$ for a map $(\rho,t)_{2}^{B}$ must have both arguments in $(t - \frac{1}{3},t + \frac{1}{3}).$
\label{lem: pairs land alpha}
\end{lemma}


\begin{figure}[htb!]
    \centering
    \includegraphics[height=2.5in]{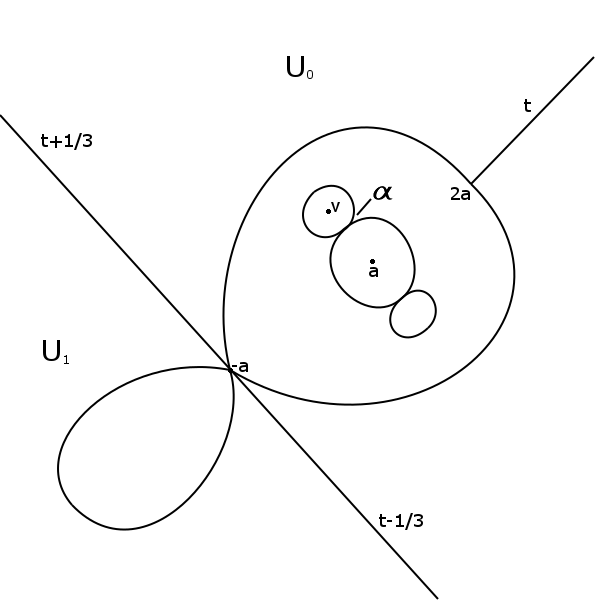}
    \caption{An illustration of a typical boundary of the maximal domain of definition of the B{\"o}ttcher coordinate for a map on the external parameter ray of argument $t$, with important dynamical external rays marked. The same applies to a map in $\E_{1}$ with a disc in place of the main basilica component illustrated.}
    \label{fig: creq updated}
\end{figure}

\begin{lemma}
At most four period two rays could possibly land at $\alpha$ for any $F \in \E_{2}^{B}$.
\label{lem: lim-land-alpha}
\end{lemma}

\begin{proof}
Any arc containing all six period two arguments will necessarily have length greater than $\frac{2}{3}$. Therefore, by Lemma \ref{lem: pairs land alpha}, at least one pair must always be excluded.
\end{proof}

We now show that any period $2$ ray that enters $U_0$ must land on $\alpha$. To do this, we will decompose $\E_{2}^{B}$ into regions depending on how many pairs of rays belong to $U_0$. 

\begin{defn}
We will define the set $W := \underset{i=1,2,3,4}\bigcup W_{i}$, where $W_{1}$ will consist of all maps in $\E_{2}^{B}$ on parameter rays of argument $t \in (\frac{1}{24},\frac{2}{24})$, and $W_{2}, W_{3},$ and $W_{4}$ are defined similarly with the intervals $(\frac{10}{24},\frac{11}{24}), (\frac{13}{24},\frac{14}{24})$, and $(\frac{22}{24},\frac{23}{24})$, respectively, see Figure \ref{fig: S2w}.
\label{def: setW}
\end{defn}

\begin{figure}[htb!]
    \centering
    \includegraphics[height=5in]{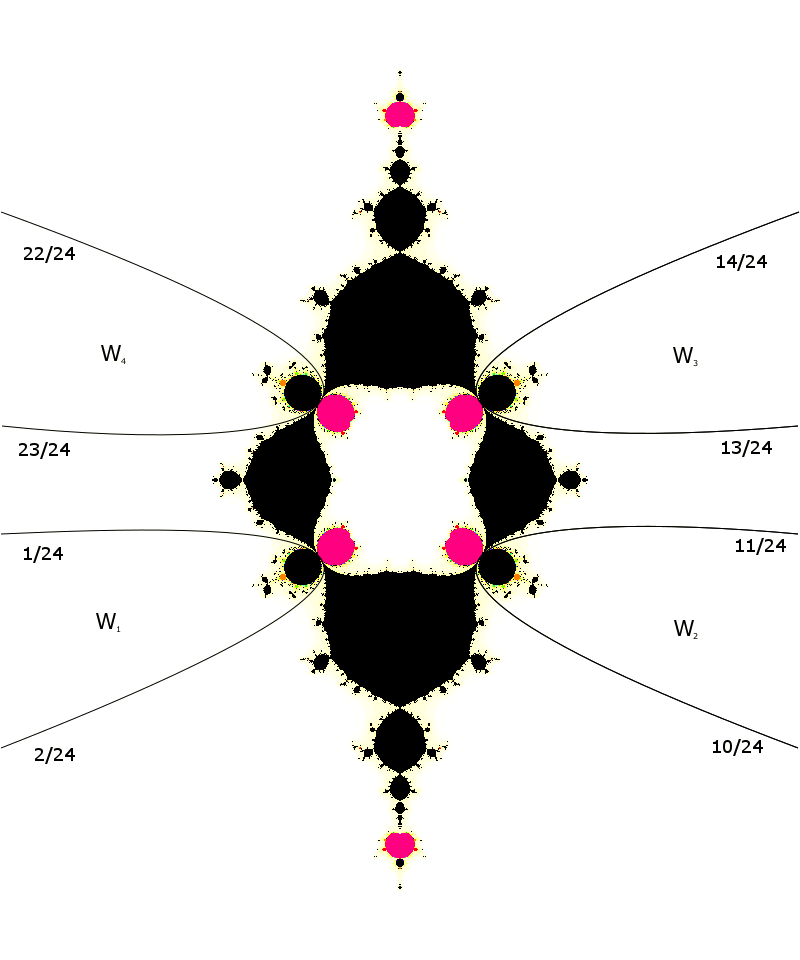}
    \caption{$\Stwo$ with the basilica escape region decomposed into regions according to the number of rays landing at the fixed point $\alpha$.}
    \label{fig: S2w}
\end{figure}

The following result explains what happens on the boundary $\partial W$. Note that this boundary consists of eight parameter rays, all of which are coperiodic of co-period two.

\begin{lemma}
Suppose $F \in R_{e}(t) \subset \E_{2}^{B}$. Then all period two dynamic rays land on $J(F)$ if and only if $t \in \R/\Z$ is not coperiodic of co-period two.
\label{lem: pertwoland}
\end{lemma}

\begin{proof}
Suppose first that $t$ is coperiodic of coperiod two. Then one of the dynamic rays $r_{e}(t + \frac{1}{3})$ and $r_{e}(t - \frac{1}{3})$ is a period two ray which crashes into the free critical point $-a$ and therefore does not land on $J(F)$. Conversely, suppose that there is a period two dynamic ray $r_{e}(\tilde{t})$ which does not land. Then there is a minimal $n \in \N$ such that $F^{-n}(-a) \in r_{e}(\tilde{t})$. Taking forward iterates, it then follows that one of the rays $r_{e}(\tilde{t})$ or $r_{e}(3\tilde{t})$, depending on the parity of $n$, crashes into the free critical point $-a$. This implies that $\tilde{t} \in \{t - \frac{1}{3}, t + \frac{1}{3}\}$ or $3\tilde{t} \in \{t - \frac{1}{3}, t + \frac{1}{3}\}$, and therefore that $t$ is coperiodic of co-period two by definition.   
\end{proof}
 
 A simple combinatorial argument and the above lemma shows that for $F = (\rho,t) \in \E_{2}^{B} \setminus W$, exactly two period two rays are contained in $U_0$. For example, if $t \in (\frac{23}{24},\frac{1}{24})$ then only the period two rays of angle $2/8$ and $6/8$ can enter $U_0$. On the other hand, if $F = (\rho,t) \in W$, then there are four period two rays entering $U_0$. An analogous result holds in $\E_1$. These are summarized in the following lemmas and illustrated in Figure~\ref{fig: S2p}.
 
 \begin{lemma}
 For $(\rho,t)_{2}^{B} \in \E_{2}^{B} \setminus W$, there is exactly one possible pair of period two rays contained in $U_0$, and for $(\rho,t)_{2}^{B} \in W$, there are exactly two pairs of period two rays which are contained in $U_0$.  \label{lem: no.raysalpha}
 \end{lemma}

\begin{lemma}
  For $(\rho,t)_{1} = \Psi^{-1}((\rho,t)_{2}^{B})$, if $(\rho,t)_{2}^{B} \in \E_{2}^{B} \setminus W$, then there is exactly one pair of period two rays contained in $U_0$, and if $(\rho,t)_{2}^{B} \in W$, then there are exactly two pairs of period two rays contained in $U_0$. 
  \label{lem: land basin of a}
  \end{lemma}

\begin{figure}[h!tb]
    \centering
    \includegraphics[height=5in]{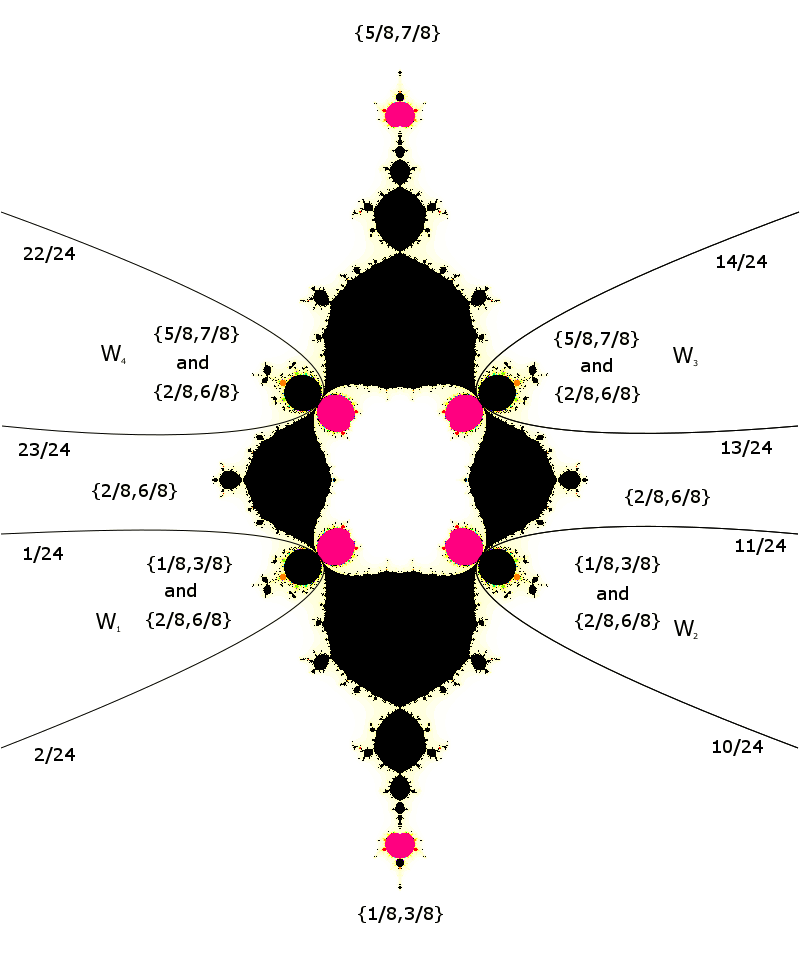}
    \caption{The pairs of rays landing at $\alpha$ in each  piece of the basilica escape region.}
    \label{fig: S2p}
\end{figure}

\begin{figure}[htb!]
    \centering
    \includegraphics[height=3.5in]{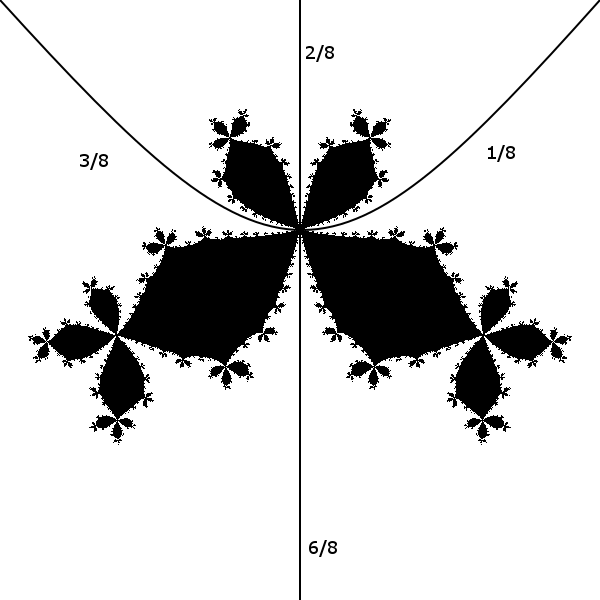} 
    \caption{The Julia set of the map at the center of the unique type $D$ component of period two in $W_{1}$. The period two external rays which land together are marked.}
    \label{fig: bowtie}
\end{figure}

We now show that if $F \in \E_2^B$, then any pair of period two rays contained in $U_0$ must land on $\alpha$. Similarly, if $F \in \E_1$, any pair of period two rays contained in $U_0$ must land on a period two cycle on the boundary of $\widehat{\A}_{a}$.


\begin{lemma}
For $(\rho,t)_{2}^{B} \in W$, let $\{t_{1},t_{2}\}$ and $\{t_{3},t_{4}\}$ be the two pairs of arguments of the rays contained in $U_0$. Then we have $\lambda_{e}(t_{i}) = \alpha$ for $i = 1,2,3,4.$
\label{lem: four rays land}
\end{lemma}

\begin{proof}
We just prove the result for $W_1$, the other cases are similar. Choose any $(\rho,t)_{2}^{B} \in W_{1}$. For this map we may take $t_{1} = \frac{1}{8}$, $t_{2} = \frac{3}{8}$, $t_{3} = \frac{2}{8}$, and $t_{4} = \frac{6}{8}$. From Theorem \ref{thm: BM}, we know that $(\rho,t)_{2}^{B}$ has the same period $2$ orbit portrait as any other map which belongs to the same component $\mathcal{F}_{k}$ of $\Tess_2(\overline{\Stwo})$ . In particular, we consider the unique post-critically finite polynomial $p$ in the period two hyperbolic component in $V_{i}$, as in Figure \ref{fig: bowtie}. For this map, the rays of arguments $t_{1}$, $t_{2}$, $t_{3}$, and $t_{4}$ all land together at a common fixed point on the boundaries of the Fatou components containing the marked critical point $a_{F}$ and its associated critical value $v_{F}$. Therefore, we conclude that for $(\rho,t)_{2}^{B}$, $\lambda_{e}(t_{i}) = \alpha$ for $i = 1,2,3,4.$
\end{proof}

\begin{figure}[htb!]
    \centering
    \includegraphics[height=3.5in]{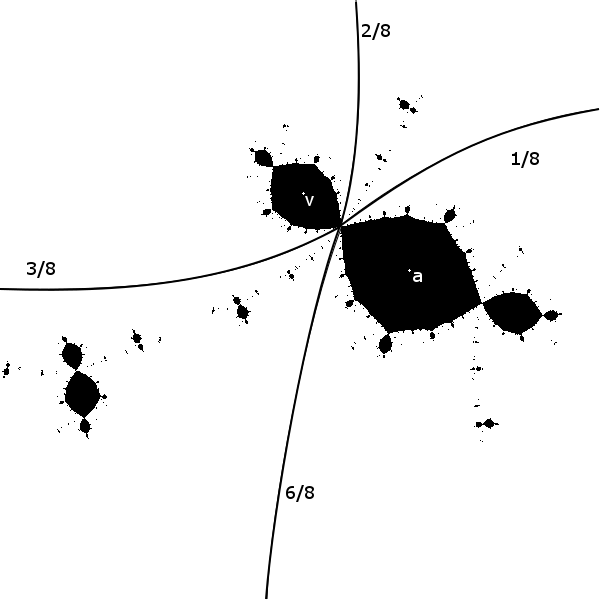} 
    \caption{A Julia set from $W_{1}$ showing the four rays landing at the fixed point $\alpha$.}
    \label{fig: w1-4}
\end{figure}

We now observe that any pair of period two dynamic rays which land at the fixed point $\alpha$ in the Julia set of a map $\Psi(F) \in \E_{2}^{B}$ corresponds with a pair of dynamical external rays for the map $F \in \E_{1}$ which land at distinct points in a two cycle on the boundary of the main disc component of the filled Julia set of $F$.

\begin{lemma}
Given $F = (\rho,t)_{1} \in \E_{1}$ and $\{s,\tilde{s}\} \in \R/\Z$ a two cycle under tripling for which $\lambda_{e}^{\Psi(F)}(s) = \lambda_{e}^{\Psi(F)}(\tilde{s}) = \alpha$ in $K(\Psi(F))$, we have $\lambda_{e}^{F}(s) \not= \lambda_{e}^{F}(\tilde{s})$ form a two cycle on $\partial \widehat{\A}_{a} \in K(F)$. 
\label{lem: land two cycle}
\end{lemma}

\begin{proof}
 There are two cases, depending on the number of rays landing on $\alpha$. First, suppose $\Psi(F) \in \E_{2}^{B} \setminus W$, so that only two external rays for $\Psi(F)$ land on $\alpha$ . Then by Lemma~\ref{lem: land basin of a} there is a unique period two cycle on $\partial \widehat{\A}_{a}$ for $F$ and these points must be $\lambda_{e}^{F}(s)$ and $\lambda_{e}^{F}(\tilde{s})$, since no other period two rays belong to $U_0$. 

For the second case, suppose $\Psi(F) \in W$, so that four external rays for $\Psi(F)$ land on $\alpha$. We will assume $\Psi(F) \in W_{1}$; the other cases follow by a similar argument.

If $\Psi(F) \in W_1$, then by Theorem \ref{thm: BM}, the period two orbit portrait of $\Psi(F)$ is the same as that for the map at the center of the unique period $2$ type $D$ component in $W_1$. The Julia set for this map is given in Figure~\ref{fig: bowtie} - we observe in particular that the rays of angles $\frac{1}{8}, \frac{2}{8}, \frac{3}{8}$ and $\frac{6}{8}$ all land together for this map. Thus the rays of angles $\frac{1}{8}, \frac{2}{8}, \frac{3}{8}$ and $\frac{6}{8}$ for $\Psi(F)$ must all land together at the point $\alpha$.

Again by Theorem \ref{thm: BM}, the period two orbit portrait of $F$ is the same as that of any other map in the period two decomposition of $\Sone$. In this component we have the $\frac{1}{8}$ and $\frac{2}{8}$ rays landing at a common point, as well as the $\frac{3}{8}$ and $\frac{6}{8}$ rays as in Figure \ref{fig: s1w1}. Therefore, we have the two cycle $$\lambda_{e}^{F}\left(\frac{1}{8}\right) = \lambda_{e}^{F}\left(\frac{2}{8}\right) \longleftrightarrow \lambda_{e}^{F}\left(\frac{3}{8}\right) = \lambda_{e}^{F}\left(\frac{6}{8}\right) \in \partial \widehat{\A}_{a}.$$ 
\end{proof}

\begin{figure}[htb!]
    \centering
    \includegraphics[height=4in]{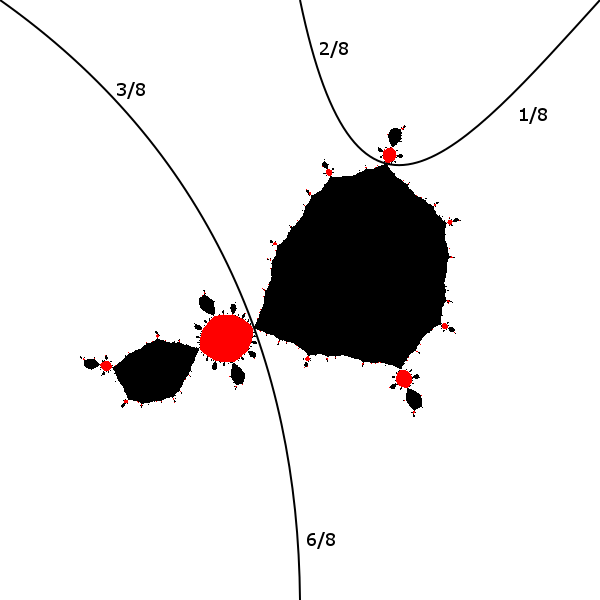}
    \caption{The Julia set of a map from the component of the period two decomposition of $\Sone$ which contains the entire preimage under $\Psi$ of $W_{1} \subset \Stwo$.}
    \label{fig: s1w1}
\end{figure}

The following result provides a combinatorial description of the behavior of $\Psi$. Informally, this can be thought of as taking each of the disc components of the filled Julia set $K(F)$ for a map $F \in \E_{1}$, and replacing them with copies of the basilica in order to obtain a topological object which is homeomorphic to $K(\Psi(F))$. Later on we will extend this result to the boundary of $\E_1$ in order to describe the behavior of $\widehat{\Psi}$.

\begin{lemma}
Given $F = (\rho,t)_{1} \in \E_{1}$, $K(\Psi(F))$ is homeomorphic to the quotient of $K(F)$ obtained by making the identifications from the lamination of the basilica along the internal rays of each disc component of $K(F)$.
\label{lem: bas-lam-psi}
\end{lemma}

\begin{proof}
Consider a pair of period two rays $r_e(t)$, $r_e(\tilde{t})$ in the same orbit landing together at the point $\alpha \in K(\Psi(F))$. By Lemma \ref{lem: land two cycle}, the corresponding external rays in $F$ of arguments $t$ and $\tilde{t}$ must land at a two cycle on the boundary of the main disc component of $K(F)$. This two cycle corresponds to the landing points of the internal rays of angles $\frac{1}{3}$ and $\frac{2}{3}$ in the main disc component of $K(F)$. Now observe that there are preimages of $r_e(t)$ and $r_e(\tilde{t})$ which land at the unique non-fixed preimage of $\alpha$ which belongs to the main basilica component of $K(\Psi(F))$. The corresponding rays must land at points on the boundary of the main disc component of $F$ whose internal arguments are $\frac{1}{6}$ and $\frac{5}{6}$. Continuing inductively, we see that the rays landing on preimages of $\alpha$ in the main basilica component of $K(\Psi(F))$ correspond to rays landing on the main disc component of $F$ whose internal arguments are given by basilica angles. This process corresponds exactly with that of applying the basilica lamination in the main disc component of $K(F)$.

The preimage components are tackled in the same way, by considering the preimages of the main disc component. Since each non-trivial component of the filled Julia set is hybrid equivalent to the basilica, it follows that there are no more identifications.\end{proof}

We now claim that the period $q$ orbit portraits for the map $F \in \Esc_1$ and $\Psi(F) \in \Esc_2^B$ are the same, save for the well-understood case when $q=2$. Given a parabolic parameter $\mathfrak{p}$, if the orbit relation $\phi \sim \psi$ is true for all maps in a neighborhood of $\mathfrak{p}$, then we call this relation a background orbit relation. The following result can be obtained from \cite{ICP2}, but we give a proof using the above results here.

\begin{lemma}
 Let $q \in \N$ and suppose that $F$ lies in a face of $\Tess_q(\overline{\Sone})$. Then the period $q$ orbit portrait for $F$ is contained in the period $q$ orbit portrait of $\Psi(F)$. In particular, the orbit portraits are equal if $q \neq 2$.
 \label{lem:same_orb_ports}
\end{lemma}

\begin{proof}
 The construction given in Lemma~\ref{lem: bas-lam-psi} can only make the portrait larger so the orbit portrait of $F$ must be contained in the orbit portrait of $\Psi(F)$. Indeed, the only periodic point which is the landing point of more than one ray produced by the process is $\alpha$, which is the landing point of (at least) two rays of period $2$. So if $q \neq 2$, the orbit portraits must be the same.
\end{proof}

\subsection{$\widehat{\Psi}$ is well-defined}

We will prove the following result.

\begin{theorem}
 Let $\theta$ and $\theta'$ be rational angles. If $R_e(\theta)$ and $R_e(\theta')$ land together on $\partial \Esc_1$, then they land together on $\partial \Esc_2^B$. 
 \label{thm:psiwelldef}
\end{theorem}

The proof makes use of the results of Branner-Hubbard \cite{ICP2} and the (unpublished) work of the first author and Milnor \cite{CP3}. Indeed, the present authors have been informed that this result can be obtained directly from \cite{ICP2}, but we could not find the explicit statement we needed in the literature. For completeness we prove the result here.

First we prove the statement above when $\theta$ and $\theta'$ are coperiodic. As a consequence, this will allow us to conclude that the tessellations $\Tess_q$ in $\Esc_1$ and $\Esc_2^B$ are essentially the same. By \cite{CP3}, we know that coperiodic rays in $\Esc_1$ and $\Esc_2^B$ always have at least one partner ray. Indeed, in $\Sone$, all coperiodic rays land in pairs. In $\Stwo$, the only exception is when there are four rays landing on the parabolic points that lie on the intersection of the boundary of the two escape regions. We will call a coperiod $q$ parameter ray a primary ray if, of all the parameter rays landing at the parabolic map $\mathfrak{p}$, it is one of the two rays which are closest to the hyperbolic component $H_\mathfrak{p}$ which has $\mathfrak{p}$ as its root point. All other rays landing at $\mathfrak{p}$ will be called secondary rays. Note that all coperiodic rays in $\Esc_1$ and $\Esc_2^B$ are primary rays.

Following \cite{CP3} (compare \cite{POER} for the quadratic case), we define a wake $W~\subset~S_p$ to be a simply connected subset of $\mathcal{S}_p$ which 
\begin{enumerate}
 \item is bounded by a a pair of parameter rays which are in the same escape region and which land at a common parabolic point $\mathfrak{p}$.
\item contains a hyperbolic component of Type $D$ which has $\mathfrak{p}$ as a boundary point.
\end{enumerate}
In \cite{CP3}, the authors make the following conjecture.

\begin{Conj}[Monotonicity Conjecture] As we cross any primary ray of co-period $q$ into the face which contains $H_\mathfrak{p}$, the period $q$ orbit portrait is replaced by a strictly larger orbit portrait.
\end{Conj}

We will assume this conjecture in proving Theorem~\ref{thm:psiwelldef}.

\begin{lemma}
Let $\theta$ and $\theta'$ be coperiodic angles. Then $R_e(\theta)$ and $R_e(\theta')$ land together on $\partial \Esc_1$ if and only if they land together on $\partial \Esc_2^B$.
\label{Lem:coperiod}
\end{lemma}

\begin{proof}
One may easily verify by inspection that the statement holds for coperiod $2$ rays. We prove that if $R_e(\theta)$ and $R_{e}(\theta')$ land together on $\partial \Esc_1$ then they land together on $\partial \Esc_2^B$. The proof of the converse is similar.

Suppose the rays $R_e(\theta)$ and $R_e(\theta')$ of coperiodic angles $\theta$ and $\theta'$ of coperiod $q \neq 2$ land at a common parabolic point $\mathfrak{p}$.  Then these rays, along with $\mathfrak{p}$, form a wake $W$, where maps in $\Esc_1 \cap W$ have parameter angle $\theta < t < \theta'$. All maps $F$ inside this wake have a period $q$ orbit portrait which contains the period $q$ orbit portrait of the parabolic map $\mathfrak{p}$. In particular, this holds for $F \in \Esc_1 \cap W$. By Lemma~\ref{lem:same_orb_ports}, since $q \neq 2$, we see that the map $F' \in \Esc_2^B$ has the same period $q$ orbit portrait as $F$. Furthermore, since $W$ must be a primary wake, then if we cross the parameter ray of angle $\theta$ transversally to enter $W$, then the period $q$ orbit portrait will get strictly larger. The same holds if we cross into $W$ by passing through the parameter ray of angle $\theta'$.

To obtain a contradiction, we assume that $R_e(\theta)$ and $R_e(\theta')$ do not land together on $\partial \Esc_2^B$. Suppose instead that $R_e(\theta)$ lands with $R_e(\alpha)$ and $R_e(\theta')$ lands with $R_e(\beta)$. We split into cases depending on the relative circular ordering of $\theta$, $\theta'$, $\alpha$ and $\beta$.

\begin{itemize}
\item $\alpha < \theta < \theta' < \beta$ or $\alpha < \theta < \beta < \theta'$. The rays $R_e(\alpha)$ and $R_e(\theta)$  bound a wake $W'$. But this means that if we cross $R_e(\theta)$ by increasing the angle, we will leave $W'$ and so the orbit portrait will get strictly smaller. However, the same action in $\Esc_1$ sees the orbit portrait grow strictly larger. This is a contradiction, since by Lemma~\ref{lem:same_orb_ports} the period $q$ orbit portraits of $F \in \Esc_1$ and $\Psi(F) \in \Esc_2^B$ should be the same.
\item $\theta < \alpha < \beta < \theta'$ or $\theta < \alpha < \theta' < \beta$. The rays $R_e(\theta)$ and $R_e(\alpha)$ bound a primary wake $W'$ which has a root point $\mathfrak{p}'$ which has the same period $q$ orbit portrait as that of $\mathfrak{p}$. Thus, as we cross the ray $R_e(\alpha)$ to leave $W'$, the orbit portrait will get strictly smaller than that of $\mathfrak{p}'$, and thus it is strictly smaller than that of $\mathfrak{p}$. But if we consider a map $F' \in \Esc_2^B$ obtained after crossing the ray $R_e(\alpha)$, the map $\Psi^{-1} \in \Esc_1$ has parameter angle $t$ that lies between $\theta$ and $\theta'$, and so has an orbit portrait which contains the orbit portrait of $\mathfrak{p}$. This is a contradiction.
\item $\theta < \theta' < \beta < \alpha$ or $\theta < \beta < \theta' < \alpha$.  In this case the wake $\widetilde{W}$ bounded by $R_e(\theta')$ and $R_e(\beta)$ is a subwake of the wake $W'$ bounded by $R_e(\theta)$ and $R_e(\alpha)$. This means that as we cross $R_e(\theta')$ and enter $\widetilde{W}$, we will obtain a orbit portrait strictly larger than that which we obtain if we cross $R_e(\theta)$ to enter $W'$. But since the two rays land together in $\Sone$, the two orbit portraits must be the same. Hence we have a contradiction.
\end{itemize}

Thus all possibilities for when $R_e(\theta)$ and $R_e(\theta')$ don't land together on $\partial \Esc_2^B$ lead to a contradiction. Hence $R_e\theta$ and $R_e(\theta')$ must land together on $\partial \Esc_2^B$.
\end{proof}

The above result implies there is a bijective correspondence between the faces of the tessellation $\Tess_q(\overline{\Sone})$ which intersect $\Esc_1$ and the faces of the tessellation $\Tess_q(\overline{\Stwo})$ which intersect $\Esc_2^B$.

\begin{cor}
 The parameter rays of angle $\theta$ and $\theta'$ belong to the same component of $\Tess^{(2)}_q(\overline{\Sone}) \cap \Esc_1$ if and only if they belong to the same component of $\Tess^{(2)}_q(\overline{\Stwo}) \cap \Esc_2^B$.
 \label{cor:samefaces}
\end{cor}

\begin{proof}
 Again we prove only one direction, since the converse is very similar. Suppose that $R_e(\theta)$ and $R_e(\theta')$ belong to the same component $\mathcal{F}_k$ of $\Tess^{(2)}_q(\overline{\Sone}) \cap \Esc_1$, but $R_e(\theta)$ and $R_e(\theta')$ do not belong to the same component of  $\Tess^{(2)}_q(\overline{\Stwo}) \cap \Esc_2^B$. Then there exist $\alpha$ and $\beta$, angles of coperiod $q$ such that the parameter rays $R_e(\alpha)$ and $R_e(\beta)$, together with their common landing point, separate the rays $R_e(\theta)$ and $R_e(\theta')$ in $\Esc_2^B$. But then these rays would have to separate the rays $R_e(\theta)$ and $R_e(\theta')$ in $\Sone$, which is impossible since they both belong to the same component $\mathcal{F}_k$.
\end{proof}

By the corollary, we see that the map $\Psi$ induces a natural bijection between the components of $\Tess_q(\overline{\Sone}) \cap \Esc_1$ and the components of $\Tess_q(\overline{\Stwo}) \cap \Esc_2^B$. We again use the notation $\Psi$ for this induced map.

We now turn our attention to the non-coperiodic case. By \cite{CP3}, if $\theta$ is rational but not coperiodic, then the parameter ray lands on a critically finite non-hyperbolic parameter.

\begin{lemma}
Suppose the rays $R_e(\theta)$ and $R_e(\theta')$ land at the same critically finite map on $\partial \Esc_1$. Then rays $R_\theta$ and $R_{\theta'}$ land together on $\partial \Esc_2^B$.
\label{Lem:critfin}
\end{lemma}

\begin{proof}
Since $R_e(\theta)$ and $R_e(\theta')$ land together on $\partial \Esc_1$ at a parameter $F_1$, then for any $q$ both these rays must belong to the same component $\mathcal{F}_k$ as each other in $\Tess_q(\overline{\Sone}) \cap \Esc_1$. Note that since we are in $\Sone$, there are no orbit relations coming from the quasiconformal copies of the disk in the filled Julia set of maps in $F_k$. The dynamical rays $r_e^{F_1}(\theta)$ and $r_e^{F_1}(\theta')$ land at the cocritical point of $F_1$, and this point eventually maps onto a repelling cycle of some period $n$.  Each point on the repelling cycle must be the landing point of at least two rays, and thus this orbit appears in the period $n$ orbit portrait for all maps in  $\mathcal{F}_k$. In particular, by taking preimages we see that the rays $r_e(\theta)$ and $r_e(\theta')$ land together in the dynamical plane for all maps in $\mathcal{F}_k$.

Since the landing points of the rays $R_e(\theta)$ and $R_e(\theta')$ are in the same face $\mathcal{F}_k$ of $\Sone$, then Corollary~\ref{cor:samefaces} implies they must belong to the same component $\mathcal{F}'_j = \Psi(\mathcal{F}_k)$ in $\Tess_q(\overline{\Stwo}) \cap \Esc_2^B$. Furthermore, the period $n$ orbit portrait in $\mathcal{F}_k$ is a subset of that in $\mathcal{F}'_j$ (in fact, if $n \neq 2$,  the orbit portraits are equal). With the same argument as in the previous paragraph, we see that the dynamical rays $r_e^{F'}(\theta)$ and $r_e^{F'}(\theta')$ land together in the dynamical plane for all maps in $F' \in \mathcal{F}'_j$. In particular, for $F_2'$, the landing point of $R_e(\theta)$ in $\partial \Esc_2^B$, the ray $r_e^{F_2'}(\theta)$ lands on the cocritical point together with the ray $r_e^{F_2'}(\theta')$. But since $r_e^{F_2'}(\theta)$ and $r_e^{F_2'}(\theta')$ land on the cocritical point of $F_2'$, the parameter rays $R_e(\theta)$ and $R_e(\theta')$ in $\Esc_2^B$ land on $F_2$.
\end{proof}

\begin{Rem}
 Note that, unlike with the previous results, the converse of the above is not true (indeed, the exact cases of when it fails are discussed in the rest of the paper). The reason for this is that the quasiconformal copies of the basilica for maps in $\Esc_2^B$ introduce extra orbit relations.
\end{Rem}

\begin{proof}[Proof of Theorem~\ref{thm:psiwelldef}]
 The theorem follows from Lemmas~\ref{Lem:coperiod} and \ref{Lem:critfin}.
\end{proof}

The following is immediate.

\begin{cor}
  The extension of $\Psi$ given by $\widehat{\Psi}((1,t)_1) = (1,t)_2^B$ for $t$ rational is well-defined and continuous.
 \label{lem: psi-land}
 \end{cor}

\section{Main results}\label{sec:main}

We begin with our main results by extending the result of Lemma \ref{lem: bas-lam-psi} to the boundary to obtain a similar result describing how to obtain $K(\widehat{\Psi}(F))$ topologically from $K(F)$. To begin, we show that the maps in $\partial \E_{2}^{B}$ which are landing points of parameter rays have a distinguished fixed point. We first deal with the case where these rays are not coperiodic of co-period two.

\begin{lemma}
Suppose $t \in \R/\Z$ is not coperiodic of co-period two. The map $(1,t)_{2}^{B} \in \partial \E_{2}^{B}$ has a unique fixed point in $J$ on the boundary of the Fatou components containing the marked critical point $a$, and its associated critical value $v$. Furthermore, the rays which land at this fixed point are exactly those which land at the fixed point $\alpha$ for any map $(\rho,t)_{2}^{B} \in R_{e}(t) \subset \E_{2}^{B}$.    
\label{lem: alpha bndry}
\end{lemma}

\begin{proof}
Suppose $t$ is not coperiodic of co-period two. The landing point $(1,t)_{2}^{B}$ is in the same component of the period two decomposition as the rest of the ray $R_{e}(t)$. Therefore, by Theorem \ref{thm: BM}, the dynamic rays which land together at $\alpha$ for any map along the parameter ray $R_{e}(t)$ still land together at a fixed point in the landing map $(1,t)_{2}^{B}$.  \end{proof}

We will continue to refer to this unique fixed point as $\alpha$ for maps on the boundary just as we did for maps in the escape region.

The following result extends that of Lemma \ref{lem: bas-lam-psi} to the boundary of $\E_1$. The first part deals with landing points of parameter rays which are not coperiodic. For such maps, the free critical point $-a$ is in the Julia set, and every component of the filled Julia set is a component of the basin of attraction of the marked critical point $a$. The second part then deals with landing points of coperiodic parameter rays, which are parabolic maps. For these maps, $-a$ belongs to a parabolic basin, and the filled Julia set consists of the basin of attraction of $a$ as well as the parabolic basins.  


\begin{theorem}
Given $F = (1,t)_{1} \in \partial\E_{1}$.
\begin{enumerate}
\item If $R_{e}(t)$ is not coperiodic, $K(\widehat{\Psi}(F)) = K((1,t)_{2}^{B})$ is homeomorphic to the quotient of $K(F)$ obtained by making the identifications from the lamination of the basilica along the internal rays of every Fatou component of $K(F)$. 
\item If $R_{e}(t)$ is coperiodic, $K(\widehat{\Psi}(F)) = K((1,t)_{2}^{B})$ is homeomorphic to the quotient of $K(F)$ obtained by making the identifications from the lamination of the basilica along the internal rays of the Fatou components of $K(F)$ which make up the basin of attraction $\A_{a_{F}}$ of the marked critical point $a_{F}$ for $F$.
\end{enumerate}
\label{thm: thebigone}
\end{theorem}

\begin{proof}
Let $F = (1,t)_{1} \in \partial\E_{1}$.
\begin{enumerate}
\item Assume that $R_{e}(t)$ is not coperiodic. By Lemma \ref{lem: alpha bndry} we have the fixed point $\alpha \in K(\widehat{\Psi}(F))$. Similarly to Lemma \ref{lem: bas-lam-psi}, there are one or two pairs of period two dynamic external rays landing together at $\alpha$, for which the corresponding dynamic rays for $F$ do not land together. We describe this as an identification made by $\widehat{\Psi}$, and we find other identifications by taking preimages as we did previously in Lemma \ref{lem: bas-lam-psi}. Thus we conclude that the identifications from the lamination of the basilica along the internal rays of every Fatou component of $K(F)$ exist in $K(\widehat{\Psi}(F))$. We want to show that these are the only identifications made by $\widehat{\Psi}$. Since $R_{e}(t)$ is not coperiodic, neither map $F, \widehat{\Psi}(F)$ is a parabolic map, see \cite{CP3}. Therefore, the free critical point is in the Julia set for both of these maps. Since each attracting cycle must attract a critical point, and the free critical point is in the Julia set, the maps must each have one attracting cycle, namely that of the marked critical point. It then follows that every bounded Fatou component of $F$ is a preimage of the main disc component, and every Fatou component of $\widehat{\Psi}(F)$ eventually maps to the two cycle of the Fatou components containing the marked critical point and its associated critical value. Therefore, the only identifications made by $\widehat{\Psi}$ are those which come from putting the lamination of the basilica in every disc component of $K(F)$. 
\item Now, suppose that $R_{e}(t)$ is coperiodic. It follows that $F$ and $\widehat{\Psi}(F)$ are parabolic maps, see \cite{CP3}. We again have the identifications arising from preimages of $\alpha \in K(\widehat{\Psi}(F))$ which are described by putting the lamination of the basilica in each component of $K(F)$. Every other Fatou component of $F$ is then a component of the parabolic basin. In both $K(F)$ and $K(\widehat{\Psi}(F))$, we must have the rays $r_{e}(t - \frac{1}{3})$ and $r_{e}(t + \frac{1}{3})$ landing at the parabolic point in the same component of the parabolic basin as the free critical point $-a$. Since $-a$ is the only critical point which could be attracted to the parabolic orbit because of the periodicity of $a$, this then determines the period of the parabolic point, and we conclude that no further identifications are made, i.e. the rays which land together on the boundary of the parabolic basin in $K(F)$ are exactly those which land together on the boundary of the parabolic basin in $K(\widehat{\Psi}(F))$. 
\end{enumerate}
\end{proof}

We now look to apply the dynamical information of Theorem \ref{thm: thebigone} to prove our main results in parameter space. We will begin with the easier of our two cases, where we consider a map on the boundary of a type C component of $\Sone$. The following result describes where to find pairs of distinct maps on the boundary of a type C component of $\Sone$ which will be mapped by $\widehat{\Psi}$ to the same image in $\Stwo$.

\begin{theorem}
Suppose $F = \Lambda_{i}(t)$ on the boundary of a type C component of $\Sone$. 
\begin{enumerate}
\item If $t$ is a basilica angle with partner  $\tilde{t}$, let $\tilde{F} = \Lambda_{i}(\tilde{t})$ on the boundary of the same type C component. Then we have $\widehat{\Psi}(F) = \widehat{\Psi}(\tilde{F})$.
\item If $t$ is not a basilica angle, then $F$ is the only preimage under $\widehat{\Psi}$ of $\widehat{\Psi}(F)$.
\end{enumerate}
In particular, $\widehat{\Psi}$ is not injective on the boundary of type $C$ components precisely at the points whose internal arguments are basilica angles.
\label{thm: mainC}
\end{theorem}

\begin{proof}
Let $F$ be on the boundary of a type C component of $\Sone$ with internal argument $t$. Suppose there are $n$ external parameter rays landing at $F$, with arguments $\{t_{1},t_{2},...,t_{n}\}$\footnote{Conjecturally, there are at most two such parameter rays.}. Then in the dynamical plane, the $n$ external rays of arguments $\{t_{1},...t_{n}\}$ all land at the cocritical point $2a_{F}$. The point $2a_{F}$ is on the boundary of a Fatou component of $F$, and the argument of the internal ray in this Fatou component landing at this point is $t$. Now, for $G := \widehat{\Psi}(F) \in \Stwo$, we still have the $n$ external rays of arguments $\{t_{1},...,t_{n}\}$ landing at the cocritical point $2a_{G}$ since $\Psi$ preserves arguments of parameter rays, and by Corollary \ref{lem: psi-land}, $\widehat{\Psi}$ sends the landing point of the parameter ray of argument $t$ in $\Sone$ to the landing point of the parameter ray of argument $t$ in $\Stwo$. If $t$ is not a basilica angle, then it follows from Theorem \ref{thm: thebigone} that these are the only external arguments for rays landing at $2a_{G}$. Therefore, we conclude that these are the only arguments for parameter rays landing at $G$, and $\widehat{\Psi}^{-1}(G) = \{F\}$. However, if $t$ is a basilica angle with partner $\tilde{t}$, then by Theorem \ref{thm: thebigone} the arguments of the external rays for $F$ landing on the component whose boundary contains $2a_{F}$ at internal argument $\tilde{t}$ will be arguments of external rays for $G$ which also land at $2a_{G}$. This means that $G$ is the landing point of parameter rays of arguments $t$ and $\tilde{t}$ in $\Stwo$. We may then conclude that $\widehat{\Psi}^{-1}(\widehat{\Psi}(F)) = \{F,\tilde{F}\}$. 
\end{proof}

We will now focus on the principal hyperbolic component $\phyp$ in $\Sone$, or in other words, the unique type A component. This case is more difficult due to the double cover of internal arguments in this component. It will be convenient to split this component up into four quadrants, which will be separated by the four internal rays of arguments  $\frac{1}{3}$ and $\frac{2}{3}$. We will label them so that the first quadrant contains on its boundary the landing map of the external ray of argument zero, and the rest will be labeled in the counterclockwise direction as usual, see Figure \ref{fig: S1quad}. We will refer to Quadrants I and III, and similarly II and IV, as opposite quadrants. We will use these quadrants to give a useful notation for internal rays in $\phyp$. By $R_{i}^{\mathrm{I}}(t)$, we will mean the internal ray of argument $t$ which lies in quadrant I. Note that we need to be sure that $t \in (\frac{1}{3},\frac{2}{3})$ for $R_{i}^{\mathrm{I}}(t)$ to exist. We will take care of the internal rays which bound the quadrants by choosing the odd numbered quadrant which it bounds, i.e. $R_{i}^{\mathrm{I}}(\frac{1}{3}), R_{i}^{\mathrm{I}}(\frac{2}{3}), R_{i}^{\mathrm{III}}(\frac{1}{3}),$ and $R_{i}^{\mathrm{III}}(\frac{2}{3})$ will denote these four internal rays in $\phyp$.    

\begin{figure}[htb!]
    \centering
    \includegraphics[width=3.5in]{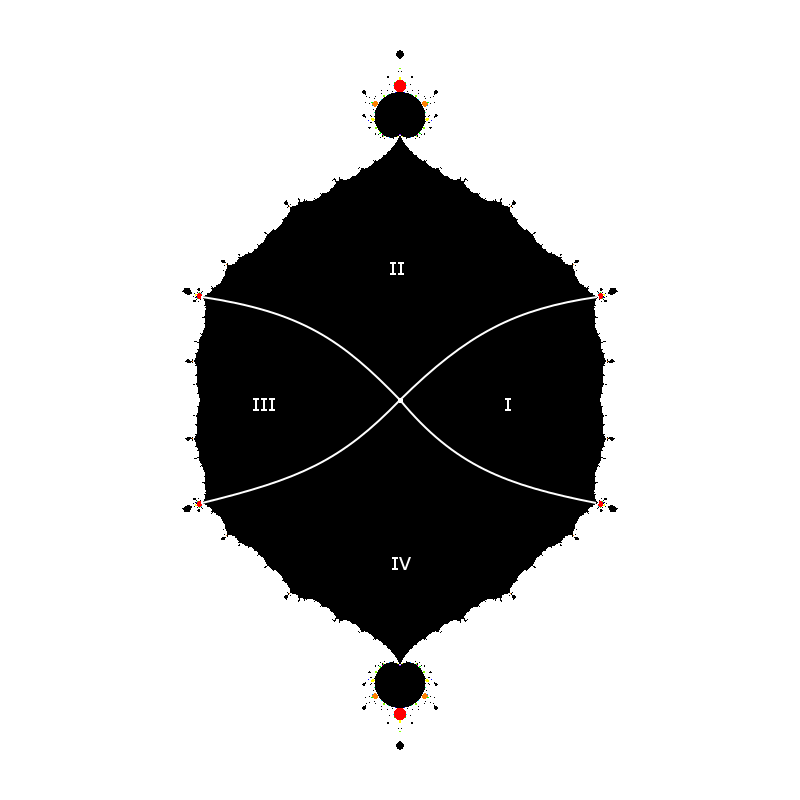}
    \caption{The principal hyperbolic component of $\Sone$ split up into four quadrants.}
    \label{fig: S1quad}
\end{figure}

Let $\widehat{\Lambda}_{i}(t)$ denote the two element set containing the landing maps of both of the internal rays in $\phyp$ of argument $t$. The following result describes which maps on $\partial \phyp$ which are not landing points of a ray of argument $\frac{1}{3}$ or $\frac{2}{3}$ will share a common image under $\widehat{\Psi}$ with another map from $\partial \phyp$.   

\begin{lemma}
Suppose $t \not\in \{\frac{1}{3},\frac{2}{3}\}$, and let $F \in \widehat{\Lambda}_{i}(t) \subset \partial\phyp$. 
\begin{enumerate}
\item If $t$ is a basilica angle with partner $\tilde{t}$, then there is one map from $\widehat{\Lambda}_{i}(\tilde{t})$ in the same quadrant of $\phyp$ as $F$, and one in the opposite quadrant. For at least one of these maps $\tilde{F} \in \widehat{\Lambda}_{i}(\tilde{t})$, we have $\widehat{\Psi}(\tilde{F}) = \widehat{\Psi}(F)$. 

\item If $t$ is not a basilica angle. Then $F$ is the only preimage of $\widehat{\Psi}(F)$.
\end{enumerate}
\label{lem: mainA}
\end{lemma}

\begin{proof}
Suppose $t \not\in \{\frac{1}{3},\frac{2}{3}\}$ and $F \in \widehat{\Lambda}_{i}(t) \subset \partial\phyp$. First, let $t$ be a basilica angle with partner $\tilde{t}$. Since leaves in the lamination of the basilica do not cross the minor leaf connecting $\frac{1}{3}$ and $\frac{2}{3}$, together with the fact that opposite quadrants contain rays of the same range of arguments, we conclude that there is one map from $\widehat{\Lambda}_{i}(\tilde{t})$ in the same quadrant of $\phyp$ as $F$, and one in the opposite quadrant. The case for a basilica angle follows similarly to the proof of Theorem \ref{thm: mainC}.    
\end{proof}

We conclude with the final results. Theorem \ref{thm: mainA} fully describes the action of $\widehat{\Psi}$ on $\phyp \subset \Sone$. Taking Theorem \ref{thm: mainA} together with Lemma \ref{lem: mainA}, we establish for the unique type A component $\phyp$ what Theorem \ref{thm: mainC} established for each type C component of $\Sone$, with the exception of the four parabolic maps on this boundary which are landing points of co-period two external parameter rays. These parabolic maps are taken care of in Theorem \ref{thm: mainParabolic}.     

\begin{theorem}
Suppose $t \not\in \{\frac{1}{3},\frac{2}{3}\}$ is a basilica angle with partner $\tilde{t}$, let $F \in \widehat{\Lambda}_{i}(t)$, and write $\widehat{\Lambda}_{i}(\tilde{t}) = \{\tilde{F}_{s}, \tilde{F}_{o}\}$ where $\tilde{F}_{s}$ is in the same quadrant as $F$, and $\tilde{F}_{o}$ is in the opposite  quadrant. Then $\widehat{\Psi}^{-1}(\widehat{\Psi}(F)) = \{F,\tilde{F}_{s}\}$.
\label{thm: mainA}
\end{theorem}

\begin{proof}
For a contradiction, suppose that $F \in \widehat{\Lambda}_{i}(t)$ where $t \not\in \{\frac{1}{3},\frac{2}{3}\}$ is a basilica angle with partner $\tilde{t}$, and that $\widehat{\Psi}^{-1}(\widehat{\Psi}(F)) = \{F,\tilde{F}_{o}\}$. Note that by Lemma \ref{lem: mainA}, only the two cases $\widehat{\Psi}^{-1}(\widehat{\Psi}(F)) = \{F,\tilde{F}_{s}\}$ or $\widehat{\Psi}^{-1}(\widehat{\Psi}(F)) = \{F,\tilde{F}_{o}\}$ are possible. First, assume that $F$ is on the boundary of quadrant I in $\Sone$. It then follows that the argument $s$ of the external parameter ray landing at $F$ is in $\left(\frac{23}{24},\frac{1}{24}\right)$, and the argument $\tilde{s}$ of the external parameter ray landing at $\tilde{F}_{o}$ is in $\left(\frac{11}{24},\frac{13}{24}\right)$. Now, define $G := \widehat{\Psi}(F) = \widehat{\Psi}(\tilde{F}_{o}) \in \Stwo$. Then we know that $\lambda_{e}(\frac{2}{8}) = \lambda_{e}(\frac{6}{8}) = \alpha_{G} \in K(G)$, see Figure \ref{fig: S2p}. However, we must also have $\lambda_{e}(s) = \lambda_{e}(\tilde{s}) = 2a_{G} \in K(G)$. This is a contradiction, since the crossing of rays would violate the injectivity of the B{\"o}ttcher coordinate. Therefore, we have $\widehat{\Psi}^{-1}(\widehat{\Psi}(F)) = \{F,\tilde{F}_{s}\}$, as desired. The remaining cases follow similarly, with $\lambda_{e}(\frac{5}{8}) = \lambda_{e}(\frac{7}{8}) = \alpha_{G} \in K(G)$ when $F$ is assumed in quadrant II, and $\lambda_{e}(\frac{1}{8}) = \lambda_{e}(\frac{3}{8}) = \alpha_{G} \in K(G)$ when $F$ is assumed in quadrant IV. 
\end{proof}

\begin{figure}[htb!]
    \centering
    \includegraphics[width=2.5in]{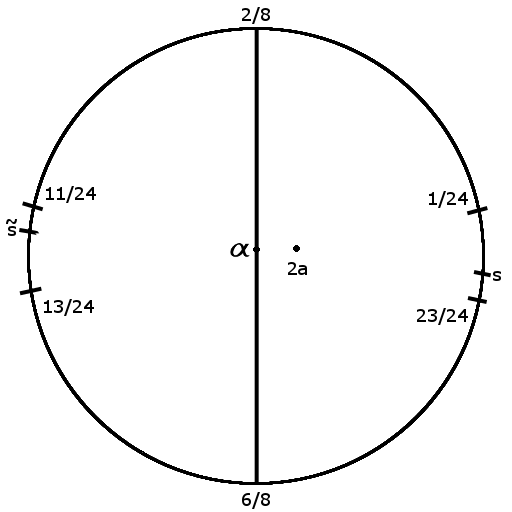}
    \caption{An illustration of the contradiction from the proof of Theorem \ref{thm: mainA} when $F$ is taken in the boundary of quadrant I of $\phyp$.}
    \label{fig: contra}
\end{figure}

We finally need to take into account those basilica angles for which the previous theorem does not apply.

\begin{theorem}
Suppose $t \in \{\frac{1}{3},\frac{2}{3}\}$, and let $F \in \widehat{\Lambda}_{i}(t) \subset \partial\phyp$, then $F$ is a parabolic map for which $\widehat{\Psi}^{-1}(\widehat{\Psi}(F)) = \{F\}$.  
\label{thm: mainParabolic}
\end{theorem}

\begin{proof}
The four landing points of internal rays of $\phyp$ of arguments $\frac{1}{3}$ and $\frac{2}{3}$ are landing points of coperiodic external rays of co-period two. From \cite{CP3}, we know that such landing points are necessarily parabolic maps. Since $\widehat{\Psi}$ sends the landing point of the parameter ray of argument $t$ in $\Sone$ to the landing point of the parameter ray of argument $t$ in $\Stwo$ by Corollary \ref{lem: psi-land}, it follows that these four distinct parabolic maps correspond bijectively under $\widehat{\Psi}$ with the four distinct parabolic maps in $\Stwo$ on the boundaries of both escape regions which are the landing points of the corresponding coperiodic rays of co-period two.    
\end{proof} 

 We now have all the ingredients we need to prove our main result.
 
 \begin{proof}[Proof of Main Theorem A]
  The case for Type $C$ is taken care of in Theorem~\ref{thm: mainC} and the Type $A$ component is dealt with in Theorems~\ref{thm: mainA} and \ref{thm: mainParabolic}. 
\end{proof}

\bibliographystyle{plain}
\bibliography{idents}

\end{document}